\documentclass[1p]{elsarticle}

\usepackage{latexsym,amssymb,amsmath,amsthm,amsfonts,amscd}
\usepackage{enumerate}
\usepackage{graphicx}
\usepackage{tikz}
\usetikzlibrary{matrix,arrows}
\usepackage{multicol}
\usepackage{multirow}
\usepackage{tabularx}
\usepackage{listings}
\usepackage{color}
\usepackage[colorlinks]{hyperref}

\newtheorem{Theorem}{Theorem}[section]

\newtheorem{Proposition}[Theorem]{Proposition}
\newtheorem{Lemma}[Theorem]{Lemma}
\newtheorem{Corollary}[Theorem]{Corollary}

\newtheorem{Example}[Theorem]{Example}

\theoremstyle{definition}

\title{Computing sandpile configurations using integer linear programming}

\author[CA]{Carlos A. Alfaro}
\ead[CA]{carlos.alfaro@banxico.org.mx, alfaromontufar@gmail.com}
\address[CA]{
Banco de M\'exico,
Ciudad de M\'exico, M\'exico.
}
\author[CV]{Carlos E. Valencia}
\ead[CV]{cvalencia@math.cinvestav.edu.mx, cvalencia75@gmail.com}
\author[CA]{Marcos~C.~Vargas}
\ead{marcos.vargas@banxico.org.mx}
\fntext[CA,CV]{The authors were partially supported by SNI}
\address[CV]{
Departamento de
Matem\'aticas\\
Centro de Investigaci\'on y de Estudios Avanzados del IPN\\
Apartado Postal 14--740 \\
07000 Ciudad de M\'exico, M\'exico.
}

\begin{document}


\begin{abstract}
The set of recurrent configurations of a graph together with a sum operation form the sandpile group.
It is well known that recurrent sandpile configurations can be characterized as the optimal solution of certain optimization problems.
In this article, we present two new integer linear programming models, one that computes recurrent configurations and other that computes the order of the configuration.
Finally, by using duality of linear programming, we are able to compute the identity configuration for the cone of a regular graph.
\end{abstract}


\maketitle


\section{Introduction}

The {\it Abelian sandpile model} was firstly studied by Bak, Tang and Wiesenfeld \cite{Bak87,Bak88} on integer grid graphs.
It was the first example of a {\it self-organized critical system}, which attempts to explain the occurrence of power laws in many natural phenomena ~\cite{Bak96} ranging on different fields like geophysics~\cite{geoph}, optimization ~\cite{optimi1,optimi2}, economics~\cite{Biggs99} and neuroscience~\cite{brain}.
An exposure to self-organized-critically is provided in~\cite{Bak96}, and one for the sandpile model can be found in~\cite{WhatIs}.
An excellent reference on the theory of sandpiles is the book of Klivans~{\cite{Klivans}, which we follow closely.


In the graph case, the dynamics of the Abelian sandpile model~\cite{Dhar90} begins with a graph with a special vertex, called {\it sink}.
We will simply use the term graph to mean a finite connected multidigraph without loops and with a global sink.
A global sink is a vertex such that there exists a directed path from every non-sink vertex to the sink.
Note that any vertex is a global sink in a connected simple graph.
Let $G = (V,E)$ be a graph with $V$ its vertex set, $E$ its edge set and $q\in V$ be a \emph{global sink} of $G$.
For simplicity, $\widetilde{V}$ will denote the set of non-sink vertices.
We denote by $\mathbb{N}$ the set of non-negative integers.
In the sandpile model, a configuration on $(G,q)$ is a vector ${\bf c}\in \mathbb{N}^{\widetilde{V}}$, in which the entry ${\bf c}_v$ is associated with the number of {\it grains of sand or chips} placed on vertex $v$.
A non-sink vertex $v$ is called \textit{stable} if ${\bf c}_v$  is less than its {\it degree} $d_G(v)$, and {\it unstable}, otherwise.
Moreover, a configuration is called \textit{stable} if every non-sink vertex is stable.
The {\it toppling rule} in the dynamics consists on selecting an unstable vertex $u$ and moving $d_G(u)$ grains of sand from $u$ to its neighbors, in which each neighbor $v$ receives $m_{(u,v)}$ grains of sand, where $m_{(u,v)}$ denotes the number of edges going from $u$ to $v$.

The \textit{Laplacian matrix} $\Delta(G)$ of a graph $G$ is the matrix whose $(u,v)$-entry is defined by
\[
\Delta(G)_{u,v}=
\begin{cases}
d_G(u) & \text{ if } u=v,\\
-m_{(u,v)} & \text{ otherwise.}
\end{cases}
\]
The {\it reduced Laplacian matrix} $\Delta(G)_q$ is the matrix defined as the Laplacian matrix where the column and row associated with the sink are removed.
For simplicity, we will use $\Delta_q$ to denote the reduced Laplacian.
Then, toppling vertex $v_i$ in the configuration ${\bf c}$ corresponds to the subtraction the $i$-{\it th} row to ${\bf c}$, in other words, ${\bf c}-{\bf e}_i\Delta_q$.
In this setting the sink never topples, it only collects all grains of sand getting out of the dynamics of the model.
This model of configurations on $G$ together with the toppling rule is denoted by $ASM(G,q)$.

Starting with any unstable configuration and toppling unstable vertices repeatedly, we will always obtain a stable configuration after a finite sequence of topplings, see \cite[Lemma 2.4]{Holroyd}.
It is important to know that the stabilization of an unstable configuration is unique, see~\cite[Theorem 2.1]{Meester}.
The stable configuration obtained from ${\bf c}$ will be denoted by $s({\bf c})$.
An example of stabilization of a configuration is given in Figure~\ref{fig:C_5}.
The finite sequence of topplings required to get the stable configuration is known as {\it avalanche}, and its {\it size} is the number of topplings.
Then, $s({\bf c})={\bf c}-\beta\Delta_q$ for some unique $\beta\in \mathbb{N}^{\widetilde{V}}$.
\begin{center}
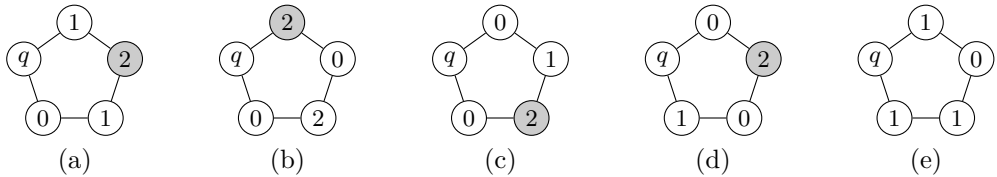
\begin{figure}[h]
  \begin{center}
	\begin{tabular}{c@{\extracolsep{1cm}}c@{\extracolsep{1cm}}c@{\extracolsep{1cm}}c@{\extracolsep{1cm}}c}
	\begin{tikzpicture}[scale=0.7] 
		\tikzstyle{every node}=[circle, inner sep=2pt, minimum width=5pt] 
		\path (18:1) node[draw, fill=black!20] (v1) {\small $2$};
    		\path (90:1)  node[draw] (v2) {\small $1$};
    		\path (162:1) node[draw] (v3) {\small $q$};
    		\path (234:1) node[draw] (v4) {\small $0$};
    		\path (306:1) node[draw] (v5) {\small $1$};
    		\draw (v1) -- (v2) -- (v3) -- (v4) -- (v5) -- (v1);
	\end{tikzpicture}
		&
	\begin{tikzpicture}[scale=0.7] 
		\tikzstyle{every node}=[circle, inner sep=2pt, minimum width=5pt] 
		\path (18:1) node[draw] (v1) {\small $0$};
    		\path (90:1)  node[draw, fill=black!20] (v2) {\small $2$};
    		\path (162:1) node[draw] (v3) {\small $q$};
    		\path (234:1) node[draw] (v4) {\small $0$};
    		\path (306:1) node[draw] (v5) {\small $2$};
    		\draw (v1) -- (v2) -- (v3) -- (v4) -- (v5) -- (v1);
	\end{tikzpicture}
		&
	\begin{tikzpicture}[scale=0.7] 
		\tikzstyle{every node}=[circle, inner sep=2pt, minimum width=5pt] 
		\path (18:1) node[draw] (v1) {\small $1$};
    		\path (90:1)  node[draw] (v2) {\small $0$};
    		\path (162:1) node[draw] (v3) {\small $q$};
    		\path (234:1) node[draw] (v4) {\small $0$};
    		\path (306:1) node[draw, fill=black!20] (v5) {\small $2$};
    		\draw (v1) -- (v2) -- (v3) -- (v4) -- (v5) -- (v1);
	\end{tikzpicture}
  &
	\begin{tikzpicture}[scale=0.7] 
		\tikzstyle{every node}=[circle, inner sep=2pt, minimum width=5pt] 
		\path (18:1) node[draw, fill=black!20] (v1) {\small $2$};
    		\path (90:1)  node[draw] (v2) {\small $0$};
    		\path (162:1) node[draw] (v3) {\small $q$};
    		\path (234:1) node[draw] (v4) {\small $1$};
    		\path (306:1) node[draw] (v5) {\small $0$};
    		\draw (v1) -- (v2) -- (v3) -- (v4) -- (v5) -- (v1);
	\end{tikzpicture}
  &
	\begin{tikzpicture}[scale=0.7] 
		\tikzstyle{every node}=[circle, inner sep=2pt, minimum width=5pt] 
		\path (18:1) node[draw] (v1) {\small $0$};
    		\path (90:1)  node[draw] (v2) {\small $1$};
    		\path (162:1) node[draw] (v3) {\small $q$};
    		\path (234:1) node[draw] (v4) {\small $1$};
    		\path (306:1) node[draw] (v5) {\small $1$};
    		\draw (v1) -- (v2) -- (v3) -- (v4) -- (v5) -- (v1);
	\end{tikzpicture}\\
  (a) & (b) & (c) & (d) & (e)
	\end{tabular}
  \end{center}
\caption{Starting in (a) with the unstable configuration ${\bf u}=(1,2,1,0)$ on the cycle graph with 5 vertices, the stable configuration ${\bf v}=(1,0,1,1)$ is reached in (e) after an avalanche of size 4. At each step the vertex to be toppled is highlighted. Note that at step $(b)$ we could select the other vertex with 2 grains to be toppled, if we continue the avalanche by toppling this vertex, then the stable configuration obtained will be the same.}
\label{fig:C_5}
\end{figure}
\end{center}

Let $\mathrm{deg}_{(G,q)}$ denote the vector $\left(\mathrm{deg}_{G}(u)\right)_{u\in \widetilde{V}}$.
A configuration ${\bf c}$ is \textit{recurrent or critical} if there exists a configuration ${\bf d}\geq \mathrm{deg}_{(G,q)}$ such that ${\bf c}=s({\bf d})$.
The sum of two configurations ${\bf c}$ and ${\bf d}$ is performed entry by entry, that is, $({\bf c}+{\bf d})_u={\bf c}_u+{\bf d}_u$ for all $u\in \widetilde{V}$.
Moreover, let 
\[
{\bf c}\oplus {\bf d}:=s({\bf c}+{\bf d}).
\]
Recurrent configurations play a central role in the dynamics of the ASM, and together with the $\oplus$ operation they form an Abelian group known as \textit{sandpile group}, which is denoted by $SP(G,q)$.

\begin{Theorem}\cite[Corollary 2.16]{Holroyd}.
Let $G$ be a graph with global sink $q$, then $SP(G,q)$ is an Abelian group.
\end{Theorem}

The sandpile group $SP(G,q)$ is isomorphic to the quotient of $\mathbb{Z}^{\widetilde{V}}$ modulo the $\mathbb{Z}$-module generated by the rows of the reduced Laplacian matrix of $G$.
Thus two vectors ${\bf c}$ and ${\bf d}$ are (firing) equivalent in $\mathbb{Z}^{\widetilde{V}}$, denoted by ${\bf c}\sim{\bf d}$, if there exists a vector ${\bf z}\in \mathbb{Z}^{\widetilde{V}}$} such that ${\bf c} ={\bf d} + {\bf z}\Delta_q$.
The equivalence class of an element ${\bf c}$ is denoted by $[{\bf c}]$.
Moreover, there exists an unique recurrent configuration for every equivalence graph, see~\cite[Theorem 2.6.6]{Klivans}.
Several properties might be clearer if we consider recurrent configurations as representatives of equivalence classes of $SP(G,q)$.
Therefore, by the fundamental theorem of finitely generated Abelian groups, 
\[
SP(G,q)\cong \mathbb{Z}_{f_1}\oplus\mathbb{Z}_{f_2}\oplus\cdots\oplus\mathbb{Z}_{f_{n-1}},
\] 
where $f_i>0$ and $f_i|f_{j}$  for all $i<j$.
The algebraic structure of the sandpile group can be computed using the Smith normal form of the reduced Laplacian matrix. 
%

It is important to note that the algebraic structure of the sandpile group does not depend on the sink vertex. 
However, the set of recurrent configurations of $G$ do depends on the sink vertex.
In general, it is easier to describe the abstract structure of the sandpile group than to
give an explicit description of recurrent configurations and their generated subgroups.

The set of recurrent configurations and their generated subgroups has been described only for a few family of graphs.
For instance, in~\cite{BorgneIdentity},~\cite{Identity}, and~\cite{DartoisIdentity}
is given a partial characterization of the recurrent configuration that plays the role of the identity for the grid graph.
And it is still open to prove that the identity element of the grid graph has a large square of constant value centered at the origin \cite[Conjecture 5.7.3]{Klivans}.
In a parallel way, the sandpile dynamics can be generalized to non-singular $M$-matrices instead of Laplacian matrices, see~\cite{GuzmanKlivans}.

In this article we give an integer linear program model whose optimal solution computes the recurrent configuration equivalent to a given configuration and another one that computes its order.
This model is simpler and useful on computing the identity and the generators of the sandpile group.
For simplicity, we give the results in terms of the Laplacian matrix, but it can be extended to $M$-matrices.

\section{Recurrent configurations as optimal solutions of a integer linear program}

It is not surprising that recurrent configurations can be seen as solutions of a variety of optimization problems.
For instance in~\cite[Theorem 2.36]{Alfaro}, the unique recurrent configuration that belongs to a class of a stable configuration was characterized as the optimal solution of an integer linear program.

A configuration ${\bf c}$ is superstable if and only if $\sigma_{max}-{\bf c}$ is recurrent, where $\sigma_{max}=\mathrm{deg}_{(G,q)}-{\bf 1}$ is the maximum stable configuration.
In~\cite{Potential}, superstable configurations were characterized as the configurations that minimize the energy.
The {\it energy} $E({\bf c})$ of ${\bf c}$ is given by $\|\Delta_q^{-1}{\bf c}\|_2^2$, where $\|{\bf c}\|_2^2={\bf c}\cdot {\bf c}$.
Thus, superstable configurations are the solutions of the
{\it energy minimization problem:}
Given a configuration ${\bf c}$, determine a configuration of lowest energy that is firing equivalent to ${\bf c}$.
That is, a solution to:

\begin{eqnarray}
{\sf argmin} & E({\bf c})\nonumber\\
 & {\bf c}\sim{\bf d}\\
 & {\bf d}\geq {\bf 0}\nonumber
\end{eqnarray}

A nice presentation of the topic can also be found in \cite[Section 2.6.4]{Klivans}.
In~\cite{GuzmanKlivans} the idea of energy minimizers was formulated for a wide variety of objective functions under linear restrictions.

Over all optimization problems that characterize the recurrent configurations as its solution perhaps the simplest and powerful is the \emph{Least Action Principle}.

\begin{Lemma}[Least Action Principle]\cite[Proposition 2.6.18]{Klivans} \label{algrec3}
Let $G$ be a graph with global sink $q$, ${\bf c}$ be a configuration of $(G,q)$, ${\bf d} = s({\bf c})$ and ${\bf z}\in \mathbb{N}^{\widetilde{V}}$ such that 
\[
{\bf d} = {\bf c} - {\bf z}\Delta_q.
\]
If ${\bf f} = {\bf c} - {\bf w}\Delta_q$ is stable for some ${\bf w}\in \mathbb{Z}^{\widetilde{V}}$, then ${\bf z\leq w}$.
\end{Lemma}

From the Least Action Principle we get the following characterization of the stabilization of a configuration using an integer linear program.

\begin{Corollary}\label{integer2}
Let $G$ be a graph with global sink $q$, ${\bf c}$ be a configuration of $(G,q)$ and $\sigma_{max}$ be the maximum stable configuration.
If ${\bf x}^*$ is an optimal solution of the integer linear program
\begin{eqnarray}\label{model1}
\text{\rm maximize }   & & {\bf1}\cdot{\bf x}^t \nonumber\\
\text{\rm subject to } & & {\bf 0}\leq {\bf c}- {\bf x}\Delta_q \leq \sigma_{max},\\
	                       & & {\bf x}\in \mathbb{Z}^{\widetilde{V}}, \nonumber
\end{eqnarray}
then ${\bf x}^*$ is unique and $s({\bf c})={\bf c}-{\bf x}^* \Delta_q$.
\end{Corollary}
\begin{proof}
Let ${\bf f}={\bf c}- {\bf x}^*\Delta_q$ and $s({\bf c})={\bf c} - {\bf z}\Delta_q$.
Since ${\bf f}$ is stable, by the Least Action Principle ${\bf z}\leq {\bf x}^*$.
Moreover, since ${\bf x}^*$ is an optimal solution of~(\ref{model1}), then ${\bf z}= {\bf x}^*$ and therefore ${\bf x}^*$ is unique.
\end{proof}

Note that Corollary~\ref{integer2} applies to any configuration.
However, a special case is given when ${\bf c}\geq \mathrm{deg}_{(G,q)}$, for which we get recurrent configurations.
Dually, we have an equivalent principle, which we call the Maximal Action Principle.

\begin{Lemma}[Maximal Action Principle] \label{maximal}
Let $G$ be a graph with global sink $q$, ${\bf c}$ be a stable configuration of $(G,q)$, ${\bf d}$ be the recurrent configuration firing equivalent to ${\bf c}$ and ${\bf z}\in \mathbb{N}^{\widetilde{V}}$ such that 
\[
{\bf d} = {\bf c} + {\bf z}\Delta_q.
\]
If ${\bf f} = {\bf c} + {\bf w}\Delta_q$ is stable for some ${\bf w}\in \mathbb{Z}^{\widetilde{V}}$, then ${\bf w} \leq {\bf z}$.
\end{Lemma}
\begin{proof}
Since $\Delta_q$ is a $M$-matrix there exists ${\bf h}\in \mathbb{N}_+^{\widetilde{V}}$ such that ${\bf h}\Delta_q\geq {\bf 1}$.
Thus, it is not difficult to check that there exists $t\in \mathbb{N}_+$ such that ${\bf c}'={\bf c}+t{\bf h}\Delta_q\geq \mathrm{deg}_{(G,q)}$.

Now, let ${\bf z}'\in \mathbb{N}^{\widetilde{V}}$ such that $s({\bf c}') = {\bf c}' - {\bf z}'\Delta_q$.
Since ${\bf c}'\in [{\bf c}]$, by the definition of recurrent configuration, 
\[
{\bf d}=s({\bf c}')={\bf c} + (t{\bf h}-{\bf z}')\Delta_q
\]
and ${\bf z}=t{\bf h}-{\bf z}'$.
Since ${\bf f} = {\bf c} + {\bf w}\Delta_q$ if and only if ${\bf f} = {\bf c}' - (t{\bf h}-{\bf w})\Delta_q$, then applying the Least Action Principle to ${\bf c}'$ we get that $t{\bf h}-{\bf z}={\bf z}'\leq t{\bf h}-{\bf w}$ and we get the result.
\end{proof}

In a similar way, we get a characterization of a recurrent configuration as a solution of an integer linear program using the Maximal Action Principle.

\begin{Theorem}\cite[Theorem 2.36]{Alfaro}\label{integer1}
Let $G$ be a graph with global sink $q$, ${\bf c}$ be a stable configuration of $(G,q)$.
If ${\bf x}^*$ is an optimal solution of the integer linear program 
\begin{eqnarray}\label{main:model}
\text{\rm maximize }   & & {\bf1}\cdot{\bf x} \nonumber\\
\text{\rm subject to } & & {\bf 0}\leq {\bf c}+{\bf x}\Delta_q \leq \sigma_{max},\\
	                       & & {\bf x}\in {\mathbb Z}^{\widetilde{V}}, \nonumber
\end{eqnarray}
then ${\bf x}^*$ is unique and ${\bf c} + {\bf x}^* \Delta_q \in SP(G,q)$.
\end{Theorem}
\begin{proof}
It follows by similar arguments of those used in Corollary~\ref{integer2} but using the Maximal Action Principle instead the Least Action Principle.
\end{proof}

Clearly applying Theorem~\ref{integer1} to ${\bf c} ={\bf 0}$ we can compute the identity of $SP(G,q)$.

\begin{Corollary}\label{identidad}
Let $G$ be a graph with global sink $q$.
If ${\bf x}^*$ is an optimal solution of the integer linear problem:
\begin{eqnarray}\label{identity:model}
\text{\rm maximize }   & & {\bf1\cdot x} \nonumber\\
\text{\rm subject to } & & {\bf 0}\leq {\bf x} \Delta_q \leq \sigma_{max},\\
	                       & & {\bf x}\in \mathbb{Z}^{\widetilde{V}}, \nonumber
\end{eqnarray}
then ${\bf x}^*\Delta_q$ is the identity in $SP(G,q)$.
\end{Corollary}
\begin{proof}
Since ${\bf 0}$ is stable and the identity of $\mathbb{Z}^{\widetilde{V}}$, then by Theorem~\ref{integer1}, ${\bf x}^* \Delta_q$ is recurrent and therefore the identity in $SP(G,q)$.
\end{proof}

In a similar way, the generators of $SP(G,q)$ can also be computed using Theorem~\ref{integer1} by taking ${\bf c}$ as the canonical vectors $\{ {\bf e}_i \}$.
Moreover, the degree in the group $SP(G,q)$ of any recurrent configuration can be computed by the following integer linear program.
Recall that the order of an element ${\bf c}$ of a group is the smallest positive integer $d$ such that $d{\bf c}={\bf e}$, where ${\bf e}$ denotes the identity element of the group.

\begin{Theorem}\label{coro:order}
Let $G$ be a graph with global sink $q\in V(G)$ and ${\bf c}$ be a configuration.
If ${\bf x}^*$ is an optimal solution of the integer linear problem:
\begin{eqnarray}\label{order:model}
\text{\rm minimize } & & d\nonumber\\
\text{\rm subject to }
                                   &  & {\bf x} \Delta_q = d{\bf c} ,\\
                                   &  & d\in \mathbb{N}_+, \nonumber\\
	                       & & {\bf x}\in \mathbb{Z}^{\widetilde{V}},\nonumber
\end{eqnarray}
then $d$ is the order of $[{\bf c}]$ in $SP(G,q)$.
\end{Theorem}
\begin{proof}
It follows because $[{\bf 0}]=[{\bf c}]$ if and only if ${\bf x} \Delta_q = {\bf c}$ and $d$ is the minimum positive integer such that $d[{\bf c}] =[d{\bf c}] = [{\bf 0}]$.
\end{proof}

Next example illustrates the previous results.

\begin{Example}
Let $C_5$ be the cycle with 5 vertices and global sink $q$ as in Figure~\ref{fig:C_5}.
By computing the Smith normal form of $\Delta_q(C_5)$, we know that 
\[
SP(C_5,q)\cong \mathbb{Z}_{5}.
\]
Now, let ${\bf e}_i$ denotes the vector with a 1 in the $i$-{\it th} vector and 0's elsewhere.
The solution of the integer linear program (\ref{main:model}) given in Theorem~\ref{integer1} with ${\bf c}$ equal to a vectors ${\bf e}_i$ or ${\bf 0}$ gives us the recurrent configurations for $SP(C_5,q)$.

In a similar way, the solution of the integer linear program (\ref{order:model}) given in Theorem~\ref{coro:order} gives us their orders, see Table~\ref{table1}.

\begin{table}[h!]
\begin{center}
\begin{tabular}{|c|c|c|c|}
\hline
${\bf c}$ & ${\bf x}^*$ & ${\bf c}+{\bf x}^*\Delta_q$ & order\\
\hline
(0,0,0,0) & (2,3,3,2) & (1,1,1,1) & 1\\
(1,0,0,0) & (1,2,2,1) & (1,1,1,0) & 5\\
(0,1,0,0) & (1,1,1,1) & (1,1,0,1) & 5\\
(0,0,1,0) & (1,1,1,1) & (1,0,1,1) & 5\\
(0,0,0,1) & (1,2,2,1) & (0,1,1,1) & 5\\
\hline
\end{tabular}
\end{center}
\caption{Recurrent configurations of $SP(C_5,q)$ and their orders.}
\label{table1}
\end{table}
\end{Example}

The interpretation of the recurrent configurations of a graph as optimal solution of a integer linear programming gives us a powerful concept of duality.
The dual problem of the relaxation of the integer linear program (\ref{identity:model}) is given by:
\begin{eqnarray}
\text{\rm minimize } & & \begin{bmatrix}\sigma_{max}^t & {\bf 0}^t \end{bmatrix}{\bf y}\nonumber\\
\text{\rm subject to } & & 
\begin{bmatrix}
\Delta_q & \Delta_q
\end{bmatrix}
{\bf y} = {\bf 1}\nonumber\\
                                   &  & {\bf y}_i\geq {\bf 0} \text{ for } i = 1,\dots, n\nonumber.\\
                                   &  & {\bf y}_i\leq {\bf 0} \text{ for } i = n+1,\dots, 2n\nonumber.
\end{eqnarray} 

As a consequence of the Weak Duality Theorem~\cite[Corollary 4.2]{bertsimas} we get the following result.

\begin{Proposition}
Let ${\bf x}$ and ${\bf y}$ be feasible solutions of the relaxation of (\ref{identity:model}) and its dual, respectively.
If 
\[
\begin{bmatrix}\sigma_{max}^t & {\bf 0}^t \end{bmatrix}{\bf y} = {\bf1}^t\cdot {\bf x},
\] 
then ${\bf  x}$ and ${\bf y}$ are optimal solutions of the relaxation of (\ref{identity:model}) and its dual, respectively.
\end{Proposition}

We finish with a characterization of the identity in the special case of the cone of a $r$-regular graph.
Given a graph $G$, the cone of a graph $G$ is the graph $CG$ obtained from $G$ by adding a new apex vertex $u$ and the edges between $u$ and all the vertices of $G$.

\begin{Corollary}
Let $G$ be a $r$-regular graph.
If the global sink $q$ is the apex of $CG$, then $r\cdot {\bf 1}$ is the identity of $SP(CG,q)$.
\end{Corollary}
\begin{proof}
Let ${\bf x}=r{\bf 1}$ and ${\bf y}=\begin{bmatrix}{\bf 1}\\{\bf 0}\end{bmatrix}$.
Since $\Delta_q(CG) {\bf 1} = {\bf 1}$, $r{\bf 1}$ is a feasible solution of the linear relaxation of (\ref{identity:model}) and ${\bf y}$ is a solution of its dual.
On the other hand, their costs are equal to 
\[
r{\bf 1}\cdot {\bf 1} =r\cdot n=\sigma_{max}^t \cdot {\bf 1}.
\]
Therefore, by Corollary~\ref{identidad}, ${\bf x}$  is an optimal solution of the relaxation of the integer linear program $(\ref{identity:model})$.
Since the solution is integral, then $r{\bf 1}$ is the identity of the sandpile group of $(CG,q)$. 
\end{proof}

\section*{Acknowledgments}
The authors were partially supported by SNI and CONACyT.

%
%
%
%
%

\section*{\refname}

\bibliographystyle{abbrv}
\bibliography{biblio}

\begin{thebibliography}{10}

\bibitem{Alfaro}
C.~A. Alfaro.
\newblock On the sandpile group of a graph.
\newblock Master's thesis, Mathematics Department, Cinvestav-IPN, Mexico,
  February 19, 2010.

\bibitem{Bak96}
P.~Bak.
\newblock {\em How Nature Works}.
\newblock Copernicus, an inprint of Springer-Verlag New York, Inc., first
  edition, 1996.

\bibitem{Bak87}
P.~Bak, C.~Tang, and K.~Wiesenfeld.
\newblock Self-organized criticality: An explanation of the 1/f noise.
\newblock {\em Phys. Rev. Lett.}, 59:381--384, Jul 1987.

\bibitem{Bak88}
P.~Bak, C.~Tang, and K.~Wiesenfeld.
\newblock Self-organized criticality.
\newblock {\em Phys. Rev. A (3)}, 38(1):364--374, 1988.

\bibitem{Potential}
M.~Baker and F.~Shokrieh.
\newblock Chip-firing games, potential theory on graphs, and spanning trees.
\newblock {\em J. Combin. Theory Ser. A}, 120(1):164--182, 2013.

\bibitem{bertsimas}
D.~Bertsimas and J.~Tsitsiklis.
\newblock {\em Introduction to Linear Optimization}.
\newblock Athena Scientific, 1st edition, 1997.

\bibitem{Biggs99}
N.~L. Biggs.
\newblock Chip-firing and the critical group of a graph.
\newblock {\em J. Algebraic Combin.}, 9(1):25--45, 1999.

\bibitem{optimi1}
S.~Boettcher and A.~G. Percus.
\newblock Optimization with extremal dynamics.
\newblock {\em Complex.}, 8(2):57--62, Nov. 2002.

\bibitem{Identity}
S.~Caracciolo, G.~Paoletti, and A.~Sportiello.
\newblock Explicit characterization of the identity configuration in an abelian
  sandpile model.
\newblock {\em J. Phys. A}, 41(49):495003, 17, 2008.

\bibitem{brain}
D.~R. Chialvo.
\newblock Critical brain networks.
\newblock {\em Physica A: Statistical Mechanics and its Applications},
  340(4):756 -- 765, 2004.
\newblock Complexity and Criticality: in memory of Per Bak (1947--2002).

\bibitem{DartoisIdentity}
A.~Dartois and C.~Magnien.
\newblock Results and conjectures on the sandpile identity on a lattice.
\newblock In {\em Discrete models for complex systems, {DMCS} '03 ({L}yon)},
  Discrete Math. Theor. Comput. Sci. Proc., AB, pages 89--102. Assoc. Discrete
  Math. Theor. Comput. Sci., Nancy, 2003.

\bibitem{Dhar90}
D.~Dhar.
\newblock Self-organized critical state of sandpile automaton models.
\newblock {\em Phys.\ Rev.\ Lett.}, 64(14):1613--1616, 1990.

\bibitem{GuzmanKlivans}
J.~Guzm{\'a}n and C.~Klivans.
\newblock Chip-firing and energy minimization on {M}-matrices.
\newblock {\em J.\ Combin.\ Theory Ser.\ A}, 132:14--31, 2015.

\bibitem{optimi2}
H.~Hoffmann and D.~W. Payton.
\newblock Optimization by self-organized criticality.
\newblock {\em Scientific Reports}, 8:1--9, 2018.

\bibitem{Holroyd}
A.~E. Holroyd, L.~Levine, K.~M\'esz\'aros, Y.~Peres, J.~Propp, and D.~B.
  Wilson.
\newblock Chip-firing and rotor-routing on directed graphs.
\newblock In {\em In and out of equilibrium. 2}, volume~60 of {\em Progr.
  Probab.}, pages 331--364. Birkh\"auser, Basel, 2008.

\bibitem{Klivans}
C.~J. Klivans.
\newblock {\em The Mathematics of Chip-Firing}.
\newblock Chapman and Hall/CRC, first edition, 2019.

\bibitem{BorgneIdentity}
Y.~Le~Borgne and D.~Rossin.
\newblock On the identity of the sandpile group.
\newblock {\em Discrete Math.}, 256(3):775--790, 2002.
\newblock LaCIM 2000 Conference on Combinatorics, Computer Science and
  Applications (Montreal, QC).

\bibitem{WhatIs}
L.~Levine and J.~Propp.
\newblock What is {$\dots$} a sandpile?
\newblock {\em Notices Amer. Math. Soc.}, 57(8):976--979, 2010.

\bibitem{Meester}
R.~Meester, F.~Redig, and D.~Znamenski.
\newblock The abelian sandpile: a mathematical introduction.
\newblock {\em Markov Process. Related Fields}, 7(4):509--523, 2001.

\bibitem{geoph}
R.~F. {Smalley}, Jr., D.~L. {Turcotte}, and S.~A. {Solla}.
\newblock A renormalization group approach to the stick-slip behavior of
  faults.
\newblock {\em Journal of Geophysical Research}, 90:1894--1900, Feb. 1985.

\end{thebibliography}

\end{document}